\documentclass[reqno,12pt]{amsart}
\usepackage{a4wide}
\usepackage[T1]{fontenc}
\usepackage{amsmath,amssymb,amsthm,latexsym}
\usepackage{mathrsfs}
\usepackage[usenames,dvipsnames]{color}
\usepackage{euscript}
\usepackage{graphicx}
\usepackage{mdwlist}
\usepackage{enumerate}
\usepackage{mathtools,dsfont,wasysym}
\usepackage{stmaryrd}
\usepackage{bbm}
\usepackage[all]{xy}
\usepackage{xcolor}
\usepackage{hyperref}
\hypersetup{colorlinks=true,  linkcolor=blue, citecolor=red, urlcolor=cyan}

\newtheorem{theorem}{Theorem}[section]
\newtheorem{lemma}[theorem]{Lemma}

\newtheorem{proposition}[theorem]{Proposition}
\newtheorem{fact}[theorem]{Fact}
\newtheorem{problem}[theorem]{Problem}
\newtheorem{claim}[theorem]{Claim}
\newcounter{maintheorem}

\newtheorem{mainth}[maintheorem]{Theorem}
\theoremstyle{remark}
\newtheorem{remark}[theorem]{Remark}
\theoremstyle{definition}

\numberwithin{equation}{section}
\makeatother

\newcommand{\R}{\mathbb{R}}
\newcommand{\N}{\mathbb{N}}
\newcommand{\e}{\varepsilon}
\newcommand{\p}{\varphi}
\newcommand{\ro}{\varrho}
\renewcommand{\d}{\mathrm{d}}
\newcommand{\n}{\left\Vert\cdot\right\Vert}
\newcommand{\en}{\left\vert\cdot\right\vert}
\newcommand{\nn}[1]{{\left\vert\kern-0.25ex \left\vert\kern-0.25ex \left\vert #1 \right\vert\kern-0.25ex \right\vert\kern-0.25ex \right\vert}}
\newcommand{\cut}{\mathord{\upharpoonright}}
\renewcommand{\leq}{\leqslant}
\renewcommand{\geq}{\geqslant}

\DeclareMathOperator{\arctg}{arctg}

\DeclareFontFamily{U}{mathx}{\hyphenchar\font45}
\DeclareFontShape{U}{mathx}{m}{n}{
      <5> <6> <7> <8> <9> <10>
      <10.95> <12> <14.4> <17.28> <20.74> <24.88>
      mathx10
      }{}
\DeclareSymbolFont{mathx}{U}{mathx}{m}{n}
\DeclareFontSubstitution{U}{mathx}{m}{n}
\DeclareMathAccent{\widecheck}{0}{mathx}{"71}

\newcounter{smallromans}

\newenvironment{romanenumerate}
{\begin{list}{{\normalfont\textrm{(\roman{smallromans})}}}%
  {\usecounter{smallromans}\setlength{\itemindent}{0cm}%
   \setlength{\leftmargin}{5.5ex}\setlength{\labelwidth}{5.5ex}%
   \setlength{\topsep}{.5ex}\setlength{\partopsep}{.5ex}%
   \setlength{\itemsep}{0.1ex}}}%
{\end{list}}

\newcommand{\grad}{\nabla}
\newcommand{\On}{\mathcal{O}_n}
\newcommand{\SO}{S\mathcal{O}_2}
\newcommand{\Lip}{{{\rm Lip}_0}}

\renewcommand{\L}{\mathcal{L}}
\newcommand{\C}{{\mathcal{C}_0^1}}
\renewcommand{\P}{\mathcal{P}}
\renewcommand\qedsymbol{$\blacksquare$} 
\begin{document}
\title[Projecting Lipschitz functions onto spaces of polynomials]{Projecting Lipschitz functions\\onto spaces of polynomials}

\author[P.~H\'ajek]{Petr H\'ajek}
\address[P.~H\'ajek]{Department of Mathematics\\Faculty of Electrical Engineering\\Czech Technical University in Prague\\Technick\'a 2, 166 27 Prague 6\\ Czech Republic}
\email{hajek@math.cas.cz}

\author[T.~Russo]{Tommaso Russo}
\address[T.~Russo]{Institute of Mathematics\\ Czech Academy of Sciences\\ \v{Z}itn\'a 25, 115 67 Prague 1\\ Czech Republic; Department of Mathematics\\Faculty of Electrical Engineering\\Czech Technical University in Prague\\Technick\'a 2, 166 27 Prague 6\\ Czech Republic}
\email{russo@math.cas.cz, russotom@fel.cvut.cz}

\thanks{Research of P.~H\'ajek was supported in part by OPVVV CAAS CZ.02.1.01/0.0/0.0/16$\_$019/0000778. \\
Research of T.~Russo was supported by the GA\v{C}R project 20-22230L; RVO: 67985840 and by Gruppo Nazionale per l'Analisi Matematica, la Probabilit\`a e le loro Applicazioni (GNAMPA) of Istituto Nazionale di Alta Matematica (INdAM), Italy.}

\date{\today}

\begin{abstract} The Banach space $\mathcal{P}({}^2X)$ of $2$-homogeneous polynomials on the Banach space $X$ can be naturally embedded in the Banach space ${{\rm Lip}_0}(B_X)$ of real-valued Lipschitz functions on $B_X$ that vanish at $0$. We investigate whether $\mathcal{P}({}^2X)$ is a complemented subspace of ${{\rm Lip}_0}(B_X)$. This line of research can be considered as a polynomial counterpart to a classical result by Joram Lindenstrauss, asserting that $\mathcal{P}({}^1X)=X^*$ is complemented in ${{\rm Lip}_0}(B_X)$ for every Banach space $X$. Our main result asserts that $\mathcal{P}({}^2X)$ is not complemented in ${{\rm Lip}_0}(B_X)$ for every Banach space $X$ with non-trivial type.
\end{abstract}
\maketitle

%-----------------------------------------------------------%
%                                                           %
% 						INTRODUCTION 						%
%                                                           %
%-----------------------------------------------------------%
\section{Introduction}
Given a pointed metric space $(M,0_M)$, $\Lip(M)$ denotes the Banach space of all scalar-valued Lipschitz functions on $M$ that vanish on $0_M$. The study of Banach spaces of Lipschitz functions, and even more of their canonical preduals, the Lipschitz-free spaces, has been one of the most active fields of research within Banach space theory in the last two decades. We refer, \emph{e.g.}, to \cite{AACD, AlP, GK, Go} and \cite{CCD, CK, We, We1} for recent results and additional references on Lipschitz-free spaces and spaces of Lipschitz functions, respectively. One of the earliest results in this area is the extremely useful result by Joram Lindenstrauss \cite{Lind projection}, according to which $X^*$ is $1$-complemented in $\Lip(X)$, for every Banach space $X$ (see \cite[Chapter 7]{BL} for more details and applications).

Our main goal in this paper is to study possible extensions of the said result to higher order polynomials, in particular to the space $\P({}^2 X)$ of $2$-homogeneous polynomials on $X$. This is an instance of a well-studied line of research aiming at obtaining polynomial versions of linear results, \cite{Ar, AB, AGM, AH, FGJM, FJ, Pel, Ry}. When trying to approach the said question, we face the obvious problem that non-zero polynomials of degree at least $2$ on a Banach space $X$ are not Lipschitz functions. However, the restriction of a polynomial on $X$ to $B_X$ is a Lipschitz function and such restriction defines an isomorphic embedding of $\P({}^2X)$ in $\Lip(B_X)$ (see \S\ \ref{sec: poly} for details). Moreover, $\Lip(B_X)$ is isomorphic to $\Lip(X)$ for every Banach space $X$, \cite{Kaufmann}. Therefore, the search for the polynomial counterpart to Lindenstrauss' result leads us to the following problem.

\begin{problem}\label{Problem poly} Identify $\P({}^2 X)$ with a subspace of $\Lip(B_X)$ by restricting a polynomial on the Banach space $X$ to $B_X$. Is $\P({}^2 X)$ a complemented subspace of $\Lip(B_X)$?
\end{problem}

Let us mention in passing that a positive answer to the above problem would also give information on approximation properties for Banach spaces of Lipschitz functions. For example, if $\P({}^2 \ell_2)$ were a complemented subspace of $\Lip(B_{\ell_2})$, then $\Lip(B_{\ell_2})$ would fail the approximation property. This is because the approximation property passes to complemented subspaces, while $\P({}^2\ell_2)$ fails the approximation property \cite[p.~173]{Fl} (see also \cite{DM}). Let us also notice that, while approximation properties of Lipschitz-free Banach spaces have been widely investigated, \cite{AmP, GK, GO, Go, LP, PS}, the study of approximation properties of spaces of Lipschitz functions seems to be a largely unexplored topic, \cite{J-V}. For example, it is apparently unknown if $\Lip(B_{\ell_2})$ has the approximation property.\smallskip

The main result of our paper, Theorem \ref{mth: inf-dim} below, answers Problem \ref{Problem poly} in the negative for a rather large class of Banach spaces, including all spaces with non-trivial type, in particular all super-reflexive Banach spaces. The simplest space which is not covered by our results is $c_0$. As it is perhaps to be expected, our main theorem will be derived from its quantitative counterpart for Euclidean spaces, that reads as follows. (The precise definitions of the spaces $\P({}^2 E_n)$, $\Lip(B_{E_n})$, and $\C(B_{E_n})$ will be given in \S\ \ref{sec: prelim}.)

\begin{mainth}\label{mth: finite-dim} Let $E_n$ denote the $n$-dimensional Euclidean space $(\R^n,\n_2)$. If $Q$ is any projection from $\C(B_{E_n})$ onto $\P({}^2 E_n)$, then

$$\|Q\|\geq C\left(n-2\sqrt{2}\right)^{1/5},$$

where $C:=\frac{2}{5}\left(\frac{\sqrt{2}-1}{3}\cdot\frac{\pi}{72}\right)^{1/5}$.

In particular, the same estimate holds for every projection from $\Lip(B_{E_n})$ onto $\P({}^2 E_n)$.
\end{mainth}

The main reason why we stated the theorem in the stronger form concerning $\C(B_{E_n})$ is that the averaging argument that we shall need (see \S\ \ref{sec: average}) is simpler to describe in $\C(B_{E_n})$ rather than in the bigger space $\Lip(B_{E_n})$. In this context, it is perhaps worth observing that $\C(B_F)$ fails to be complemented in $\Lip(B_F)$, for every finite-dimensional Banach space $F$. Indeed $\C([-1,1])$ is not complemented in $\Lip([-1,1])$, since the former is isometric to $\mathcal{C}([-1,1])$ and the latter to $L_\infty([-1,1])$. By `bootstrap' from this observation, the claim for the general finite-dimensional Banach space $F$ readily follows.\smallskip

The study of projections onto a finite-dimensional subspace of a Banach space is a very classical and wide topic and Theorem \ref{mth: finite-dim} can be considered as one more result in the area. By the classical Kadets--Snobar theorem \cite{KS}, every $n$-dimensional subspace of a Banach space is $\sqrt{n}$-complemented and this is asymptotically sharp \cite{Konig}. As a non-exhaustive list of results on projection constants, let us quote \cite{Gr, Le, FS, Ba}, or the monographs \cite{TJ, Woj}. The particular case of projections onto spaces of polynomials is also very well studied (particularly in the trigonometric case), for its connections with harmonic (\cite[III.B]{Woj}, \cite{GM, OL}) and numerical (\cite{FL, TV, TV2}) analysis. Let us mention for example the classical Lozinski--Kharshiladze theorem \cite{LK} (see, \emph{e.g.}, \cite[III.B.22]{Woj}, or \cite[pp.~150--155]{Ko}) on the projection constant of the space $T^\infty_n$ of trigonometric polynomials of degree at most $n$ on the $1$-dimensional torus $\mathbb{T}$, with the sup-norm. The result claims in particular that if $P_n$ is any projection from $C(\mathbb{T})$ to $T^\infty_n$, then $\|P_n\|\geq \frac{4}{\pi^2}\log(n+1)+o(1)$. There are however three crucial differences between these classical results and our paper: we consider algebraic rather than trigonometric polynomials; we are interested in the asymptotic behaviour when the dimension of the space grows, as opposed to letting the degree of the polynomials grow; and finally we consider the Lipschitz norm instead of the sup-norm. \smallskip

Having Theorem \ref{mth: finite-dim} at our disposal, we can formulate the main result of our paper, that in particular answers in the negative Problem \ref{Problem poly}. Prior to the statement, we need to recall one definition. A Banach space $X$ is said to contain \emph{uniformly complemented} $(\ell_2^n)_{n=1}^\infty$ if there are a constant $C$ and a sequence $(F_n)_{n=1}^\infty$ of subspaces of $X$ such that, for each $n$, $F_n$ is $C$-isomorphic to $\ell_2^n$ and $C$-complemented in $X$.

\begin{mainth}\label{mth: inf-dim} If $X$ contains uniformly complemented $(\ell_2^n)_{n=1}^\infty$, then $\P({}^2 X)$ is not complemented in $\Lip(B_X)$. In particular:
\begin{romanenumerate}
    \item If $X$ has non-trivial type, $\P({}^2 X)$ is not complemented in $\Lip(B_X)$,
    \item $\P({}^2 \ell_2)$ is not complemented in $\Lip(B_{\ell_2})$.
\end{romanenumerate}
\end{mainth}

As we will see in Section \ref{Sec: inf-dim}, the first clause of Theorem \ref{mth: inf-dim} is a purely formal consequence of Theorem \ref{mth: finite-dim}. In the same section we will also extend the result to the spaces $\P({}^n X)$ of $n$-homogeneous polynomials on $X$ and $\P^n(X)$ of polynomials of degree at most $n$ (Theorem \ref{thm: higher order}). For information on type and cotype we refer to \cite{AK, DJT, MS}; here we just recall that $X$ has \emph{non-trivial type} if it has type $p$, for some $p>1$. (ii) is, of course, consequence of (i), since Hilbert spaces have type $2$. More generally Theorem \ref{mth: inf-dim} applies to every super-reflexive Banach space, as super-reflexive spaces have non-trivial type, \cite{Pisier}; on the other hand, there are non-reflexive spaces of type $2$, \cite{James}.

The statement in (i) follows from the first part of the theorem and a deep result of Figiel and Tomczak-Jaegermann, \cite{FTJ}. Indeed, it is proved in \cite{FTJ} that if the Banach space $X$ has non-trivial type, there is a constant $C$ such that, for every $\e>0$ and $n\in\N$, $X$ contains a $C$-complemented subspace that is $(1+\e)$-isomorphic to $\ell_2^n$. The result can also be found in the said monographs, \cite[Theorem 19.3]{DJT}, \cite[Theorem 15.10]{MS}; a different proof was given in \cite{BG}.\smallskip

At this stage, one might even be led to conjecture that $\P({}^2X)$ is complemented in $\Lip(B_X)$, for no infinite-dimensional Banach space $X$ (plainly, $\P({}^2X)$ is complemented in $\Lip(B_X)$ for finite-dimensional $X$). However this has long been known to be false. Indeed, back in 1976, Aron and Schottenloher \cite{AS} proved that $\P({}^k\ell_1)$ is isomorphic to $\ell_\infty$, hence even injective (for each $k\geq1$). More generally, Arias and Farmer \cite{AF} proved that $\P({}^k X)$ is isomorphic to $\ell_\infty$ whenever $X$ is a separable $\mathscr{L}_1$-space; in particular, $\P({}^k L_1)$ is complemented in $\Lip(B_{L_1})$. Since the result in \cite{AF} is stated in a somewhat different form for $\mathscr{L}_p$-spaces, $1<p<\infty$, we decided to sketch the proof of the said result in Proposition~\ref{prop: L1}. \smallskip

Our paper is organised as follows: Section \ref{sec: prelim} collects basic definitions and ancillary results that we need, while the main part of the paper is Section \ref{sec: Th A}, where Theorem \ref{mth: finite-dim} is proved. Finally, in Section \ref{Sec: inf-dim} we prove Theorem \ref{mth: inf-dim} and Proposition \ref{prop: L1}; we also mention extensions of our results to higher order polynomials.

%-----------------------------------------------------------%
%                                                           %
% 						PRELIMINARY 						%
%                                                           %
%-----------------------------------------------------------%
\section{Preliminary material}\label{sec: prelim}
Our notation is standard, \emph{e.g.}, as in \cite{AK}. We denote by $B_X$ the closed unit ball of a Banach space $X$. The notation $(E_n,\en)$ indicates the $n$-dimensional Euclidean space $(\R^n,\n_2)$; when the dimension is not important, we just indicate by $(E,\en)$, or $E$, a finite-dimensional Euclidean space. Below we gather the definitions and some basic properties of the objects we shall consider in our paper.

\subsection{Spaces of Lipschitz functions}
Let $(M,d,0_M)$ be a pointed metric space, namely a metric space $(M,d)$ with a distinguished point $0_M\in M$. The vector space $\Lip(M)$ comprising all Lipschitz functions $f\colon M\to \R$ such that $f(0_M)=0$ becomes a Banach space when endowed with the norm given by the best Lipschitz constant

$$\|f\|_\Lip := Lip(f):=\sup\left\{\frac{|f(x)-f(y)|}{d(x,y)} \colon x\neq y\in M\right\}.$$
In our paper we will only consider the cases when $M$ is a Banach space $X$ or its closed unit ball $B_X$. In either case, the distinguished point will be the origin of the Banach space and it will be denoted simply by $0$. Let us recall also here that, for every Banach space $X$, the spaces $\Lip(X)$ and $\Lip(B_X)$ are isomorphic, \cite[Corollary~3.3]{Kaufmann}. \smallskip

The open unit ball of the Euclidean space $E$ will be denoted $B_E^{\mathsf{o}}$. We denote by $\grad f$ the gradient of a differentiable function $f\colon B_E^{\mathsf{o}}\to \R$. Let

$$\C(B_E):=\left\{f\colon B_E\to \R\colon \begin{array}{c}
    f\in \mathcal{C}(B_E)\cap \mathcal{C}^1(B_E^{\mathsf{o}}), f(0)=0,\\
    \grad f \text{ is uniformly continuous on } B_E^{\mathsf{o}}
\end{array}\right\}.$$
By uniform continuity, the function $\grad f$ is bounded on $B_E^{\mathsf{o}}$ and it admits a (unique) extension, that we also denote $\grad f$, to $B_E$. Plainly, $\|f\|_\Lip=\|\grad f\|_\infty$, for every $f\in \C(B_E)$; in particular, $\C(B_E)\subseteq \Lip(B_E)$. Moreover, it is a standard result in calculus that $\C(B_E)$ is a closed subspace of $\Lip(B_E)$.

\subsection{Polynomials}\label{sec: poly}
In this section we shall briefly revise basic results on polynomials; we refer to \cite[\S~1.1]{HJ} for further details. We denote by $\L({}^n X)$ the vector space of all $n$-linear forms $M\colon X\times\dots\times X\to \R$ such that

$$\|M\|_\L:=\sup_{\substack{x_j\in B_X\\ j=1,\dots,n}} |M(x_1,\dots,x_n)|<\infty.$$
$\L({}^n X)$ is a Banach space when endowed with the norm $\n_\L$. $\L^s({}^n X)$ denotes the closed subspace of $\L({}^n X)$ comprising all symmetric $n$-linear forms.

A function $P\colon X\to\R$ is a \emph{$n$-homogeneous polynomial} if there is $M\in \L({}^n X)$ such that $P(x)=M(x,\dots,x)$ ($x\in X$). The vector space $\P({}^n X)$ of all $n$-homogeneous polynomials is turned into a Banach space when furnished with the norm

$$\|P\|_\P:= \sup_{x\in B_X}|P(x)|.$$
By a standard symmetrisation argument, for every $P\in\P({}^n X)$ there is a unique symmetric $n$-linear form $\widecheck{P} \in \L^s({}^n X)$ such that $P(x)=\widecheck{P}(x,\dots,x)$. Moreover, the form $\widecheck{P}$ satisfies

\begin{equation}\label{eq: equiv poly multi}
\|P\|_\P\leq \left\|\widecheck{P}\right\|_\L\leq \frac{n^n}{n!}\|P\|_\P.
\end{equation}
Quite importantly, however, in case of a Hilbert space $H$ one has $\|P\|_\P = \|\widecheck{P}\|_\L$ for every $P\in\P({}^n H)$, \cite[Theorem~1.16]{HJ}. These results imply that $\P({}^n X)$ and $\L^s({}^n X)$ are isomorphic Banach spaces for every Banach space $X$ and that they are indeed isometric when $X=H$ is a Hilbert space.

We shall next discuss the crucial fact mentioned in the Introduction that $\P({}^n X)$ naturally embeds into $\Lip(B_X)$, for every Banach space $X$.

\begin{fact}\label{fact: poly are Lip} Let $X$ be any Banach space and $P\in \P({}^n X)$. Then

$$n\|P\|_\P\leq \left\|P\cut_{B_X}\right\|_\Lip\leq n\left\|\widecheck{P} \right\|_\L.$$
\end{fact}

\begin{proof} For fixed $x\in B_X$, the function $[-1,1]\ni t\mapsto P(tx)= t^n P(x)$ has Lipschitz constant equal to $n \cdot|P(x)|$. Therefore, $Lip(P\cut_{B_X})\geq n \cdot|P(x)|$ and the former inequality follows. For the latter, fix $x,y\in B_X$ and write

\begin{equation*}
\begin{split}
P(x) - P(y)&= \widecheck{P}(x,\dots,x) - \widecheck{P}(y,\dots,y)\\
&= \sum_{k=1}^n \widecheck{P}(\underbrace{x,\dots,x} _\text{$k$-many},y,\dots,y) - \widecheck{P}(\underbrace{x,\dots,x} _\text{$(k-1)$-many},y,\dots,y)\\
&= \sum_{k=1}^n \widecheck{P}(x,\dots,x,\underbrace{x-y} _\text{$k$-th place},y,\dots,y).
\end{split}
\end{equation*}
Thus, $|P(x)-P(y)|\leq \sum_{k=1}^n \left\|\widecheck{P}\right\|_\L \|x\|^{k-1} \|y\|^{n-k} \|x-y\| \leq n\left\|\widecheck{P}\right\|_\L \|x-y\|.$
\end{proof}

Therefore, it follows from Fact \ref{fact: poly are Lip} and (\ref{eq: equiv poly multi}) that the map $P\mapsto P\cut_{B_X}$ defines an isomorphic embedding of $\P({}^n X)$ into $\Lip(B_X)$. In the case of a Hilbert space, we have $n\|P\|_\P= \|P\cut_{B_X} \|_\Lip$; hence, such an embedding is additionally a multiple of an isometric one. Throughout the paper, when we consider $\P({}^n X)$ as a subspace of $\Lip(B_X)$ this embedding will be the one under consideration.

\subsection{Averaging and invariant projections}\label{sec: average}
It is a standard and useful fact that projections of minimal norm tend to respect the symmetries of the Banach spaces they act on (see, \emph{e.g.}, \cite{Positselskii, Rieffel, Rudin}). In this part, we shall recall the version of this result that we need; we also sketch a proof, for the sake of completeness.\smallskip 

Let us denote by $(E,\en)$ the Euclidean space $(\R^n,\n_2)$ and by $\On$ its orthogonal group \cite[\S~1.1]{Chevalley}. Since $\On$ is a compact topological group, we can select a normalised Haar measure, which we denote $\mu$, on it \cite[Chapter~4]{HewittRoss}.

Given a projection $Q\colon \C(B_E)\to \P({}^2E)$, we shall consider the projection $\tilde{Q}\colon \C(B_E)\to \P({}^2E)$ defined by

$$\tilde{Q}(f):=\int_{\omega\in\On} Q(f\circ\omega)\circ\omega^{-1}\, \d\mu(\omega)\qquad (f\in\C(B_E)).$$
We first observe that $\tilde{Q}(f)$ is well defined as a Bochner integral of the function $\omega\mapsto Q(f\circ\omega) \circ\omega^{-1}$, from $\On$ to $\P({}^2E)$. Indeed, it suffices to check that the said function is continuous.

\begin{fact} The map $\omega\mapsto Q(f\circ\omega) \circ\omega^{-1}$ is continuous.
\end{fact}
\begin{proof} We only need to show that the map $\omega\mapsto f\circ\omega$, from $\On$ to $\C(B_E)$, is continuous at the identity matrix $\mathbbm{1}$. For every $x\in B_E$ we have

\begin{equation*}
\begin{split}
|\grad(f\circ\omega)(x)-\grad f(x)| &= |\grad f(\omega x)\omega -\grad f(x)|\\
&\leq |\grad f(\omega x)\omega -\grad f(\omega x)| +|\grad f(\omega x) -\grad f(x)|\\
&\leq \|\grad f\|_\infty\cdot \|\omega-\mathbbm{1}\| + M_{\grad f}(\|\omega -\mathbbm{1}\|),
\end{split}
\end{equation*}
where $M_{\grad f}$ denotes the modulus of continuity\footnote{We don't use the standard notation $\omega$ for a modulus of continuity as it would conflict with $\omega\in\On$.} of $\grad f$. Thus, as $\omega\to \mathbbm{1}$,

$$\|\grad (f\circ\omega- f)\|_\infty\leq \|\grad f\|_\infty\cdot \|\omega-\mathbbm{1}\| + M_{\grad f}(\|\omega -\mathbbm{1}\|) \to0.$$
\end{proof}

Once we know that $\tilde{Q}$ is well defined, it is straightforward to verify the following properties (see \cite[Theorem III.B.13]{Woj}). Property (i) claims that $\tilde{Q}$ is \emph{invariant under $\On$}.
\begin{lemma}[\cite{Rudin}]\label{lemma: average Q} $\tilde{Q}$ is a projection from $\C(B_E)$ onto $\P({}^2E)$. Moreover,
\begin{romanenumerate}
\item $\tilde{Q}(f\circ\omega)=\tilde{Q}(f)\circ\omega$ for each $f\in\C(B_E)$ and $\omega\in\On$;
\item $\|\tilde{Q}\|\leq \|Q\|$.
\end{romanenumerate}
\end{lemma}

Moreover, recall that Bochner integrals commute with bounded linear operators. Thus for every bounded linear operator $T\colon \P({}^2 E)\to Y$ one has

$$T\left(\tilde{Q}f\right)= \int_{\omega\in\On} T\left( Q(f\circ\omega)\circ\omega^{-1}\right)\, \d\mu(\omega)\qquad (f\in\C(B_E)).$$
In particular, this also yields the pointwise formula

$$\tilde{Q}(f)(x)=\int_{\omega\in\On} Q(f\circ\omega)(\omega^{-1}x)\, \d\mu(\omega)\qquad (x\in B_E).$$

Similarly as the average of a projection, we can also define the average of a function $f\in \C(B_E)$ to be the function

$$\tilde{f}:=\int_{\omega\in\On} f\circ\omega\, \d\mu(\omega).$$
As before, we see that $\tilde{f}$ is well defined as a Bochner integral. Hence $\tilde{f}\in\C(B_E)$ and

$$\grad\tilde{f}=\int_{\omega\in\On} \grad\left(f\circ\omega \right)\, \d\mu(\omega).$$
Moreover, we easily see that $Lip(\tilde{f})\leq Lip(f)$ and that $\tilde{f}$ is invariant under $\On$, in the sense that $\tilde{f}\circ \omega=\tilde{f}$ for every $\omega\in\On$. Finally, $\tilde{f}=f$ if $f$ is invariant under $\On$.

%-----------------------------------------------------------%
%                                                           %
% 						PROOF OF TH A 						%
%                                                           %
%-----------------------------------------------------------%
\section{Proof of Theorem \ref{mth: finite-dim}}\label{sec: Th A}
This section is dedicated to the proof of the finite-dimensional result Theorem \ref{mth: finite-dim}. Before entering the details of the argument, we shall sketch here the the main steps of the proof and the very idea behind it. We will construct a $\mathcal{C}^1$ function $\psi\colon B_{\R^2}\to\R$ (hence Lipschitz) which is equal to zero on a neighbourhood of the coordinate axes and whose average over all rotations of the plane is close to a `large' multiple of the polynomial $N_2:=(\n_2)^2$. Such a function can be considered as the $2$-dimensional analogue of the function $\ro\colon [-1,1]\to\R$ defined by $\ro(x):=(|x|-2\e)^2 \mathbbm{1}_{[2\e,1]}(|x|)$ ($x\in [-1,1]$); actually, the radial behaviour of $\psi$ will be determined by $\ro$.

We will then define a $\mathcal{C}^1$ function $\Psi\colon B_{\R^n}\to \R$ by means of $\psi$. We shall show that the Lipschitz constant of $\Psi$ is bounded uniformly in $n$ (but depending on $\e$), while the projection of $\Psi$ has a Lipschitz constant that grows with $n$.

\begin{proof}[Proof of Theorem \ref{mth: finite-dim}] For the sake of simplicity, we shall just denote by $E$ the $n$-dimensional Euclidean space $E_n=(\R^n,\n_2)$ and $\en$ its norm. Let

$$K_n:=\inf\left\{\|Q\|\colon Q \text{ is a projection from } \C(B_E) \text{ onto } \P({}^2E) \right\}.$$

Since $\P({}^2E)$ is a finite-dimensional Banach space, a standard compactness argument shows that $K_n$ is attained, \cite[Theorem 3]{IsbellSemadeni} (see, \emph{e.g.}, the argument in \cite[Lemma 5.17.(iii)]{FHHMZ}, or \cite[p.~12]{Lind ext cpt}). Therefore, we can pick a projection $Q$ from $\C(B_E)$ onto $\P({}^2E)$ such that $\|Q\|=K_n$. Moreover, Lemma \ref{lemma: average Q} allows us to assume additionally that $Q$ is invariant under $\On$.

We are now ready to start the first step of the proof.

\bigskip
\textbf{Step 1.} \textit{A function in the plane.}\medskip

We shall define a $\mathcal{C}^1$ function $\psi\colon B_{\R^2}\to\R$ with some symmetries and with the property that, if $\psi(x,y)\neq0$, then $|x|,|y|\geq\e$. The function $\psi$ will be defined in polar coordinates, as follows.

Fix a parameter $\e>0$ (whose value will be chosen at the end of the argument in \textbf{Step 5}). Consider the functions $\ro\colon [0,1]\to\R$ and $\tau_0\colon [0,\pi/2]\to\R$ defined by

\begin{align*}
\ro(r) &:=\begin{cases} (r-2\e)^2 & r\geq 2\e\\
0 & r<2\e \end{cases} \qquad(r\in[0,1])\\
\tau_0(\theta) &:=\max\left\{\frac{\pi}{12}-\left|\theta- \frac{\pi}{4}\right|,0 \right\} \qquad (\theta\in[0,\pi/2]).
\end{align*}

Since $\tau_0$ is not $\mathcal{C}^1$, we approximate it by a smooth function that shares some of its properties. More precisely, we fix $\delta>0$ with $\delta<\pi/72$ and we pick a $\mathcal{C}^1$ function $\tau\colon [0,\pi/2]\to\R$ such that
\begin{romanenumerate}
    \item $\tau\geq 0$ and $\tau_0- \delta\leq\tau \leq\tau_0$,
    \item $\tau$ is symmetric with respect to $\theta=\pi/4$,
    \item $Lip(\tau)\leq 1$.
\end{romanenumerate}
Such a function $\tau$ can be easily obtained by a standard convolution argument. The function $\psi$ is the function whose expression in polar coordinates in the first quadrant is given by $\ro\cdot\tau$. More precisely, we set

$$\psi(x,y):=\ro\left(\sqrt{x^2+y^2}\right)\cdot \tau\left(\arctg\left|\frac{x}{y}\right| \right) \qquad((x,y)\in B_{\R^2}).$$

\begin{fact}\label{fact: prop of psi} The function $\psi\colon B_{\R^2}\to\R$ has the following properties:
\begin{romanenumerate}
    \item $\psi(x,y)=\psi(|x|,|y|)=\psi(y,x)$,
    \item $\psi(x,y)\neq0$ only when $|x|,|y|\geq\e$,
    \item $\psi\in \C(B_{\R^2})$ and $Lip(\psi)\leq2$.
\end{romanenumerate}
\end{fact}
\begin{proof}[Proof of Fact \ref{fact: prop of psi}] \renewcommand\qedsymbol{$\square$} (i) The fact that $\psi(x,y)=\psi(y,x)$ is consequence of the symmetry of $\tau$ with respect to $\theta=\pi/4$. The equality $\psi(x,y)=\psi(|x|,|y|)$ is obvious by definition.

(ii) If $\psi(x,y)\neq0$, then $x^2+y^2\geq4\e^2$ and $\frac{1}{\sqrt{3}}|x|\leq |y|\leq \sqrt{3}|x|$. Substituting $|y|\leq \sqrt{3}|x|$ into the first inequality yields $|x|\geq\e$; similarly, one gets $|y|\geq\e$.

(iii) The fact that $\psi$ is $\mathcal{C}^1$ on $B_{\R^2}$ is clear (the smoothness in the points of the two axes follows, \emph{e.g.}, from (ii)). Thus, it suffices to show that $|\grad\psi(x,y)|\leq2$ for each $(x,y)\in B_{\R^2}$. Via the expression for the gradient in polar coordinates, we get

\begin{equation*}
\begin{split}
    |\grad\psi(x,y)|^2 &= \left(\ro'\left(\sqrt{x^2+y^2}\right)\cdot \tau\left(\arctg \left|\frac{y}{x}\right|\right) \right)^2 + \frac{1}{x^2+y^2} \left(\ro\left(\sqrt{x^2+y^2}\right)\cdot \tau'\left(\arctg \left|\frac{y}{x}\right|\right) \right)^2 \\
    &\leq (2\cdot\pi/12)^2 + \frac{1}{x^2+y^2}\left(\sqrt{x^2+y^2} \right)^2\leq \pi^2/36 +1 \leq 4.
\end{split}
\end{equation*}
\end{proof}

\bigskip
\textbf{Step 2.} \textit{The functions $\Psi$ and $\Psi_d$.}\medskip

We shall now pass to defining functions on $B_E$. Fix indices $1\leq i<j\leq n$ and define $\psi_{ij}\colon B_E\to \R$ by

$$\psi_{ij}(x):=\psi(x_i,x_j)\qquad (x=(x_1,\dots,x_n)\in B_E).$$
Next, we set $d:=\lfloor n/\sqrt{2}\rfloor$ and we define the $\mathcal{C}^1$ functions

$$\Psi:=\sum_{i<j\leq n}\psi_{ij}\qquad\text{and}\qquad \Psi_d:=\sum_{i<j\leq d}\psi_{ij}.$$

The proof will now consist in showing that the Lipschitz constants of $\Psi$ and $\Psi_d$ can be bounded uniformly in $n$ (but depending on $\e$). On the other hand, the Lipschitz constants of $Q(\Psi)$ and $Q(\Psi_d)$ grow with $n$. This latter fact will be shown in the subsequent steps, while we now show the former clause. The main reason behind it is that only `few' of the functions $\psi_{ij}$ can be nonzero at a given point, since only `few' coordinates of $x\in B_E$ can be larger than $\e$ in absolute value. 

\begin{fact}\label{fact: Lip Psi} $Lip(\Psi), Lip(\Psi_d)\leq 1/\e^4$.
\end{fact}
\begin{proof}[Proof of Fact \ref{fact: Lip Psi}] \renewcommand\qedsymbol{$\square$} Fix any $x\in B_E$, any $\xi>0$ and let $I:=\left\{i=1,\dots,n \colon |x_i|\geq \frac{\e}{1+\xi}\right\}$. Observe that $|I|\leq \left(\frac{1+\xi}{\e}\right)^2$. Indeed,

$$1\geq \sum_{i=1}^n x_i^2 \geq \sum_{i\in I} \left(\frac{\e}{1+\xi}\right)^2 = \left(\frac{\e}{1+\xi}\right)^2 \cdot|I|.$$

Now pick any $y\in B_E$ with $\|x-y\|\leq \frac{\e\xi}{1+\xi}$ and assume that $\psi_{ij}(y)\neq0$. Fact \ref{fact: prop of psi} yields $|y_i|,|y_j|\geq\e$, whence $|x_i|,|x_j|\geq \frac{\e}{1+\xi}$. Therefore, $i,j\in I$; in other words, in the ball $B(x,\frac{\e\xi}{1+\xi})$ only the functions $\psi_{ij}$ with $i,j\in I$ are different from $0$. Consequently, $\Psi$ is locally the sum of at most $\frac{1}{2}\cdot \left(\frac{1+\xi}{\e}\right)^4$ functions whose Lipschitz constant is at most $2$. It follows that $\Psi$ is $\left(\frac{1+\xi}{\e}\right)^4$-Lipschitz and letting $\xi\to0$ proves the claim for $\Psi$.

The argument for $\Psi_d$ is identical.
\end{proof}

\bigskip
\textbf{Step 3.} \textit{Computation of $Q(\Psi)$ and $Q(\Psi_d)$ and estimate of  $\max\left\{Lip(Q(\Psi)), Lip(Q(\Psi_d)) \right\}$.}\medskip

We start by giving formulas for the $2$-homogeneous polynomials $Q(\Psi)$ and $Q(\Psi_d)$ exploiting the symmetries of $\Psi$ and $\Psi_d$ and the invariance of $Q$. Since $\psi_{ij}$ is a rotation of $\psi_{12}$, it suffices to compute $Q(\psi_{12})$.

Let $(a_{ij})_{i,j=1}^n$ be scalars with $a_{ij}= a_{ji}$ and such that

$$Q(\psi_{12})(x)= \sum_{i,j=1}^n a_{ij}\cdot x_i x_j.$$

Fix $k=1,\dots,n$ and let $\omega_k$ be the reflection $(x_1,\dots,x_n)\mapsto (x_1,\dots,x_{k-1},-x_k,x_{k+1},\dots, x_n)$. Since $\psi_{12}$ is invariant under $\omega_k$, the same is true for $Q(\psi_{12})$, namely

$$\sum_{i,j=1}^n a_{ij}\cdot x_i x_j=  \sum_{\substack{i,j=1\\ i,j\neq k}}^n a_{ij}\cdot x_i x_j + a_{kk}\cdot x_k^2 -2 \sum_{\substack{i=1\\ i\neq k}}^n a_{ik}\cdot x_i x_k.$$
Therefore, $a_{ik}=0$, whenever $i\neq k$. Since $k=1,\dots,n$ was arbitrary, it follows that $a_{ij}=0$ for distinct $i,j=1,\dots,n$; hence $Q(\psi_{12})(x)= \sum_{i=1}^n a_{ii}\cdot x_i^2$.

Next, the invariance of $\psi_{12}$ under the reflection $(x_1,\dots,x_n)\mapsto (x_2,x_1,x_3,\dots, x_n)$ implies that $a_{11}=a_{22}$. Similarly, the invariance of $\psi_{12}$ under the reflection that permutes the coordinates $i$ and $j$ ($i,j=3,\dots,n$) yields that $a_{33}=\dots=a_{nn}$.

In conclusion, there are scalars $\alpha$ and $\beta$ such that

\begin{equation}\label{eq: Q psi12}
Q(\psi_{12})(x)= \alpha \left(x_1^2+x_2^2\right) +\beta \left(x_3^2+\dots + x_n^2\right).    
\end{equation}

This also implies

$$Q(\psi_{ij})(x)= \alpha \left(x_i^2+x_j^2\right) +\beta \sum_{\substack{\ell=1\\ \ell\neq i,j}}^n x_\ell^2.$$

Consequently, we can now compute

\begin{equation*}
\begin{split}
Q(\Psi)(x) &= \sum_{i<j\leq n}\left[\alpha \left(x_i^2+x_j^2\right) +\beta \sum_{\substack{\ell=1\\ \ell\neq i,j}}^n x_\ell^2 \right] \\
&= \left[\alpha(n-1) + \beta \frac{(n-1)(n-2)}{2}\right] \left(x_1^2+\dots+x_n^2\right).
\end{split}
\end{equation*}

Letting $N\in\P({}^2E)$ be the square of the norm, \emph{i.e.}, the polynomial $N(x):=x_1^2+\dots+x_n^2$, we rewrite the above as

\begin{equation}\label{eq: Q Psi}
Q(\Psi) = \left[\alpha(n-1) + \beta \frac{(n-1)(n-2)}{2}\right] \cdot N.
\end{equation}

Similarly, we compute $Q(\Psi_d)$.

\begin{equation*}
\begin{split}
Q(\Psi_d)(x) &= \sum_{i<j\leq d}\left[\alpha \left(x_i^2+x_j^2\right) +\beta \sum_{\substack{\ell=1\\ \ell\neq i,j}}^n x_\ell^2 \right] \\
&= \left[\alpha(d-1) + \beta \frac{(d-1)(d-2)}{2}\right] \left(x_1^2+\dots+x_d^2\right) + \beta \frac{d(d-1)}{2} \left(x_{d+1}^2+\dots+x_n^2\right).
\end{split}
\end{equation*}

As before, we can set $N_d(x):=x_1^2+\dots+x_d^2$ and obtain

\begin{equation}\label{eq: Q Psi d}
Q(\Psi_d) = \left[\alpha(d-1) + \beta \frac{(d-1)(d-2)}{2} \right] \cdot N_d + \beta\frac{d(d-1)}{2}(N-N_d).
\end{equation}

Since the polynomials $N$ and $N_d$ are $2$-Lipschitz on $B_E$, the equations (\ref{eq: Q Psi}) and (\ref{eq: Q Psi d}) give

\begin{align}\label{eq: Lip Q Psi}
Lip(Q(\Psi)) &= 2\left| \alpha(n-1) + \beta \frac{(n-1)(n-2)}{2} \right| \\ \label{eq: Lip Q Psi d}
Lip(Q(\Psi_d)) &\geq 2\left|\alpha(d-1) + \beta \frac{(d-1)(d-2)}{2} \right|.
\end{align}
Note that in the second estimate we used the fact that $Lip(Q(\Psi_d))\geq Lip\left((Q(\Psi_d))\cut_{\R^d}\right)$. \smallskip

We are now in a position to give a better estimate that only depends on $\alpha$ (and not on $\beta$).
\begin{claim}\label{claim: max Lip Q Psi and Psi d} The polynomials $Q(\Psi)$ and $Q(\Psi_d)$ satisfy

\begin{equation}\label{eq: max Lip Q Psi and Psi d}
    \max\left\{Lip(Q(\Psi)), Lip(Q(\Psi_d)) \right\}\geq \frac{2}{3}\left(\sqrt{2}-1\right)|\alpha|\left(n- 2\sqrt{2}\right).
\end{equation}
\end{claim}
\begin{proof}[Proof of Claim \ref{claim: max Lip Q Psi and Psi d}] \renewcommand\qedsymbol{$\square$} We fix a parameter $\lambda>0$, whose value will be determined later, and we distinguish two cases.
\begin{itemize}
    \item In case $|\beta|\geq\lambda \frac{|\alpha|}{n-2}$, from (\ref{eq: Lip Q Psi}) we have
    
    \begin{equation*}
    \begin{split}
    Lip(Q(\Psi)) &= 2\left| \alpha(n-1) + \beta \frac{(n-1)(n-2)}{2} \right|\\
    &\geq 2|\beta|\frac{(n-1)(n-2)}{2} -2|\alpha|(n-1)\\
    &\geq \lambda|\alpha|(n-1) -2|\alpha|(n-1)= (\lambda-2) |\alpha|(n-1).
    \end{split}
    \end{equation*}
    \item In case $|\beta|\leq\lambda \frac{|\alpha|}{n-2}$, on the other hand, we use (\ref{eq: Lip Q Psi d}) and we obtain
    \begin{equation*}
    \begin{split}
    Lip(Q(\Psi_d)) &\geq 2\left|\alpha(d-1) + \beta \frac{(d-1)(d-2)}{2} \right|\\
    &\geq 2|\alpha|(d-1) - 2|\beta|\frac{(d-1)(d-2)}{2}\\
    &\geq 2|\alpha|(d-1) - \lambda|\alpha|(d-1) \frac{d-2}{n-2}\\
    &\geq 2|\alpha|(d-1) - \lambda|\alpha|(d-1) \frac{1}{\sqrt{2}}\\
    &\geq |\alpha|(d-1) (2-\lambda/\sqrt{2})\\
    &\geq \left(\sqrt{2}-\lambda/2\right)|\alpha| \left(n-2\sqrt{2}\right).
    \end{split}
    \end{equation*}
    Here, we used the inequalities $\frac{d-2}{n-2}\geq 1/\sqrt{2}$ and $d-1\geq \frac{n-2\sqrt{2}}{2}$ that follow readily from our previous choice $d:=\lfloor n/\sqrt{2}\rfloor$.
\end{itemize}

The optimal choice for $\lambda$, namely $\lambda:=\frac{2}{3} \left(\sqrt{2}+2\right)$ yields $\lambda-2 = \sqrt{2} -\lambda/2=\frac{2}{3} \left(\sqrt{2}-1\right)$. Thus, we have:
\begin{itemize}
    \item In case $|\beta|\geq\lambda \frac{|\alpha|}{n-2}$,
    \begin{equation*}
    \begin{split}
    Lip(Q(\Psi)) &\geq \frac{2}{3}\left(\sqrt{2}-1\right) |\alpha|(n-1)\\
    &\geq \frac{2}{3}\left(\sqrt{2}-1\right) |\alpha| \left(n-2\sqrt{2}\right).
    \end{split}
    \end{equation*}
    \item In case $|\beta|\leq\lambda \frac{|\alpha|}{n-2}$,
    \begin{equation*}
    Lip(Q(\Psi_d)) \geq \frac{2}{3}\left(\sqrt{2}-1\right) |\alpha|\left(n-2\sqrt{2}\right).
    \end{equation*}
\end{itemize}
Therefore, the estimate (\ref{eq: max Lip Q Psi and Psi d}) is proved.
\end{proof}

\bigskip
\textbf{Step 4.} \textit{Estimate of $\alpha$.}\medskip

The estimate (\ref{eq: max Lip Q Psi and Psi d}) that we obtained in Claim \ref{claim: max Lip Q Psi and Psi d} in the previous step is still not sufficient, since we have no information on the parameter $\alpha$. It is thus the purpose of this step to handle this and to give a lower bound for $\alpha$. To wit, we prove here the following claim.

\begin{claim}\label{claim: alpha >=} Assume that $2\e K_n\leq1$. Then, the parameter $\alpha$ that appears in (\ref{eq: max Lip Q Psi and Psi d}) satisfies

\begin{equation}\label{eq: alpha >=}
    \alpha\geq \left(\frac{\pi}{72}-\delta\right)\cdot \big(1-2\e K_n\big)\geq0.
\end{equation}
\end{claim}

The proof is based on the fact, mentioned at the beginning of the present section, that the average of $\psi$ under rotations of the plane is close to a multiple of the square of the norm of $\R^2$. Therefore, its projection onto the set of $2$-homogeneous polynomials shares the same property.

\begin{proof}[Proof of Claim \ref{claim: alpha >=}] \renewcommand\qedsymbol{$\square$} We identify the group $\SO$ with the subgroup of $\On$ of rotations of ${\rm span}\{e_1,e_2\}$. Accordingly, for $\omega\in\SO$, we also write $\omega$ to denote the rotation $\omega\oplus\mathbbm{1}_{n-2}\in\On$ defined by $\omega\oplus \mathbbm{1}_{n-2} (x_1,\dots,x_n)= (\omega(x_1,x_2),x_3,\dots,x_n)$. We also denote by $\mu_2$ the Haar measure of $\SO$.

Let $\tilde{\psi}$ be the average of $\psi_{12}$ under the group $\SO$ (see \S\ \ref{sec: average}), \emph{i.e.}, let

$$\tilde{\psi}:=\int_{\omega\in \SO} \psi_{12}\circ\omega\, \d\mu_2(\omega).$$
Then for every $x\in B_E$, we have

\begin{equation*}
\begin{split}
\tilde{\psi}(x) &= \int_{\omega\in \SO} \psi_{12} \left(\omega(x_1,x_2), x_3,\dots,x_n \right) \d \mu_2(\omega) \\
&= \int_{\omega\in \SO} \psi \left(\omega(x_1,x_2) \right) \d\mu_2(\omega) \\
&= 4\cdot\int_{\theta\in [0,\pi/2]}\ro\left(\sqrt{x_1^2 + x_2^2}\right) \cdot \tau(\theta)\, \frac{\d\theta}{2\pi} \\
&=: \ro\left(\sqrt{x_1^2 + x_2^2}\right) \cdot \eta,
\end{split}
\end{equation*}
where we defined $\eta:= 4\cdot\int_{[0,\pi/2]}  \tau(\theta)\, \frac{\d\theta}{2\pi}$.

Letting, as before, $N_2\in\P({}^2E)$ be the polynomial $N_2(x)=x_1^2 +x_2^2$ and recalling that $\ro(r):=(r-2\e)^2 \cdot\mathbbm{1}_{[2\e,1]}(r)$, the above yields

$$\left(\eta N_2 -\tilde{\psi}\right)(x)= \eta\cdot \begin{cases} 
x_1^2 + x_2^2 & \text{if } x_1^2 + x_2^2\leq 4\e^2\\
4\e\sqrt{x_1^2 + x_2^2} -4\e^2 & \text{elsewhere}.
\end{cases}$$
Consequently,
\begin{equation}\label{eq: Lip psi tilde}
    Lip\left(\eta N_2 -\tilde{\psi}\right) \leq 4\e \eta.
\end{equation}

Moreover, since $\tau_0-\delta\leq \tau\leq \tau_0$ and $4\cdot\int_{[0,\pi/2]}  \tau_0(\theta)\, \frac{\d\theta}{2\pi}= \frac{\pi}{72}$, we have the estimate

\begin{equation}\label{eq: estimate on eta}
    0< \frac{\pi}{72}-\delta\leq \eta\leq\frac{\pi}{72}.
\end{equation}

On the other hand, using the facts that Bochner integrals commute with linear operators, that $Q$ is invariant under $\On$ (hence, under its subgroup $\SO$) and that $Q(\psi_{12})$ is invariant under $\SO$ by (\ref{eq: Q psi12}), we derive

\begin{equation*}
\begin{split}
Q(\tilde{\psi}) &= \int_{\omega\in \SO} Q\left(\psi_{12}\circ\omega\right)\, \d\mu_2(\omega)\\
&= \int_{\omega\in \SO} Q\left(\psi_{12}\right) \circ\omega\, \d\mu_2(\omega)\\
&= \int_{\omega\in \SO} Q\left(\psi_{12}\right)\, \d\mu_2(\omega)= Q\left(\psi_{12}\right).
\end{split}
\end{equation*}

Hence, using again $Q\psi_{12}=\alpha N_2 + \beta(N-N_2)$ from (\ref{eq: Q psi12}), we obtain

\begin{equation*}
    Q\left(\eta N_2 -\tilde{\psi}\right) = \eta N_2 -Q(\psi_{12})= \left(\eta -\alpha\right)N_2 + \beta(N-N_2),
\end{equation*}
whence

\begin{equation}\label{eq: Lip Q psi tilde}
\begin{split}
    Lip\left(Q\left(\eta N_2 -\tilde{\psi}\right)\right) &\geq Lip\left(Q \left(\eta N_2 -\tilde{\psi} \right)\cut_{\R^2} \right)\\
    &= Lip\left(\left(\eta -\alpha\right)N_2\right)= 2 \left|\alpha-\eta\right|.
\end{split}
\end{equation}

Finally, combining (\ref{eq: Lip psi tilde}) with (\ref{eq: Lip Q psi tilde}) yields

\begin{equation*}
\begin{split}
2 \left|\alpha-\eta\right| &\leq Lip\left(Q\left(\eta N_2 -\tilde{\psi}\right)\right)\\
&\leq K_n\cdot Lip\left(\eta N_2 -\tilde{\psi}\right)\\
&\leq 4\e \eta \cdot K_n.
\end{split}
\end{equation*}
Therefore, $\alpha\geq\eta \left(1-2\e K_n \right)\geq0$
whence, using (\ref{eq: estimate on eta}), the conclusion follows.
\end{proof}

\bigskip
\textbf{Step 5.} \textit{Plugging estimates together.}\medskip

In this step, we shall combine the estimates obtained in Fact \ref{fact: Lip Psi} and Claim \ref{claim: max Lip Q Psi and Psi d} with the lower bound on $\alpha$ from Claim \ref{claim: alpha >=} and we shall conclude the proof.

From Fact \ref{fact: Lip Psi} and Claim \ref{claim: max Lip Q Psi and Psi d} we deduce

\begin{equation}\label{eq: last bound}
\begin{split}
    \frac{2}{3}\left(\sqrt{2}-1\right)|\alpha|\left(n- 2\sqrt{2}\right) &\leq \max\left\{\|Q(\Psi)\|_\Lip, \|Q(\Psi_d)\|_\Lip \right\}\\
    &\leq K_n \cdot \max\left\{\|\Psi\|_\Lip, \|\Psi_d\|_\Lip \right\} \leq K_n\cdot \frac{1}{\e^4}.    
\end{split}
\end{equation}

We now set (the reason behind such a choice will be clear in a few lines)

$$c:=\frac{2}{3}\left(\sqrt{2}-1\right) \frac{\pi}{72} \qquad\text{and}\qquad \e:=\frac{2^{1/5}}{c^{1/5} \left(n-2\sqrt{2}\right)^{1/5}}$$
and we distinguish two cases, depending on whether $1-2\e K_n\geq0$ or not.

\begin{itemize}
    \item If $2\e K_n\leq1$, we see from Claim \ref{claim: alpha >=} that $\alpha\geq0$. Therefore, we can plug the estimate (\ref{eq: alpha >=}) for $\alpha$ in (\ref{eq: last bound}) and we obtain
    
    \begin{equation}\label{eq: last eq with alpha}
    \frac{1}{\e^4}\cdot K_n \geq \frac{2}{3}\left( \sqrt{2}-1\right)\cdot \left(\frac{\pi}{72}-\delta \right)\cdot(1-2\e K_n) \left(n-2\sqrt{2}\right).
    \end{equation}
    We can now let $\delta\to0$ and obtain
    
    \begin{equation*}
    \begin{split}
    \frac{1}{\e^4}\cdot K_n &\geq \frac{2}{3}\left( \sqrt{2}-1\right)\cdot \frac{\pi}{72}(1-2\e K_n) \left(n-2\sqrt{2}\right)\\
    &= c\cdot (1-2\e K_n) \left(n-2\sqrt{2}\right),
    \end{split}
    \end{equation*}
    which we rewrite as follows:
    
    \begin{equation*}
    c\cdot\left(n-2\sqrt{2}\right) \leq \left[\frac{1}{\e^4} + 2\e c \left(n-2\sqrt{2}\right) \right]K_n.    
    \end{equation*}
    The above choice of $\e$ minimises the value of the square parenthesis above and gives the value
    
    \begin{equation*}
    \left[\frac{1}{\e^4} + 2\e c \left(n-2\sqrt{2}\right) \right] = c^{4/5} \frac{5}{2^{4/5}} \left(n-2\sqrt{2}\right)^{4/5}.
    \end{equation*}
    
    Therefore, we have
    
    \begin{equation*}
    K_n \geq \frac{2^{4/5}}{5} c^{1/5} \left(n-2\sqrt{2} \right)^{1/5} = C \left(n-2\sqrt{2} \right)^{1/5}.
    \end{equation*}
    Here,
    
    \begin{equation*}
    C:= \frac{2^{4/5}}{5} c^{1/5} = \frac{2}{5} \left(\frac{\sqrt{2}-1}{3}\cdot\frac{\pi}{72} \right)^{1/5},
    \end{equation*}
    hence we obtained the inequality claimed in Theorem \ref{mth: finite-dim}.
    \item In the case $2\e K_n\geq1$, then
    
    \begin{equation*}
        K_n \geq \frac{1}{2\e} = \frac{1}{2^{6/5}} c^{1/5} \left(n-2\sqrt{2} \right)^{1/5} \geq C \left(n-2\sqrt{2} \right)^{1/5}, 
    \end{equation*}
    since $\frac{1}{2^{6/5}}\geq\frac{2^{4/5}}{5}$. Consequently, also in this case, the desired estimate is valid.
\end{itemize}
\end{proof}

%-----------------------------------------------------------%
%                                                           %
% 						CONSEQUENCES 						%
%                                                           %
%-----------------------------------------------------------%
\section{Infinite-dimensional results}\label{Sec: inf-dim}
The present section is dedicated to our infinite-dimensional results. We shall start with the proof of our main result, Theorem \ref{mth: inf-dim}. We then briefly mention the situation for separable $\mathscr{L}_1$-spaces. Finally we conclude the section discussing how to extend our results to higher order polynomials.

\begin{proof}[Proof of Theorem \ref{mth: inf-dim}] As we already observed in the Introduction, (i) follows from \cite{FTJ} and the first part of the theorem, while (ii) is a particular case of (i). We now prove the first clause.

Let $X$ be a Banach space that contains uniformly complemented $(\ell_2^n)_{n=1}^\infty$; towards a contradiction, assume that there exists a projection $Q$ from $\Lip(B_X)$ onto $\P({}^2X)$. By definition, we may pick a sequence $(F_n)_{n=1}^\infty$ of $C$-complemented subspaces of $X$ such that $F_n$ is $C$-isomorphic to $\ell_2^n$ ($n\in\N$). Let $P_n$ be projections from $X$ onto $F_n$ with $\|P_n\|\leq C$ and let $T_n \colon \ell_2^n \to F_n$ be isomorphisms with $\|T_n\|\leq1$ and $\|T_n^{-1}\|\leq C$, for each $n\in\N$. 

We define a projection $Q_n$ from $\Lip(C^2\cdot B_{\ell_2^n})$ onto $\P({}^2\ell_2^n)$ by the following formula. 

$$Q_n(f):=\left(Q(f\circ T_n^{-1}\circ (P_n\cut_{B_X}) \right)\!\cut_{F_n} \circ T_n\qquad (f\in\Lip(C^2\cdot B_{\ell_2^n})).$$

$$\xymatrix{ \Lip(C^2\cdot B_{\ell_2^n}) \ar[d]^{Q_n} \ar[rr]^{(T_n^{-1}\circ (P_n\cut_{B_X}) )^\sharp} && \Lip(B_X) \ar[d]^Q \\
\P({}^2\ell_2^n) & \P({}^2F_n) \ar[l]_{T_n^\sharp} &\P({}^2 X) \ar[l]_{\cdot\cut_{F_n}}}$$
Here, $T_n^\sharp$ is defined by $T_n^\sharp g:=g\circ T_n$ (and similarly for $(T_n^{-1}\circ (P_n\cut_{B_X}))^\sharp$). Notice that in this case we consider $\P({}^2\ell_2^n)$ as a subspace of $\Lip(C^2\cdot B_{\ell_2^n})$ via the embedding $P\mapsto P\cut_{C^2\cdot B_{\ell_2^n}}$.

Since $\|T_n^{-1}\circ P_n\|\leq C^2$, $f\circ T_n^{-1}\circ (P_n\cut_{B_X})$ is indeed a Lipschitz function on $B_X$, when $f\in \Lip(C^2\cdot B_{\ell_2^n})$. Thus, $Q(f\circ T_n^{-1}\circ (P_n\cut_{B_X}))$ is a polynomial on $X$ whose Lipschitz constant on $B_X$ is at most $C^2\|Q\|\cdot Lip(f)$. Hence, its Lipschitz constant on $C^2\cdot B_X$ is bounded by $C^4\|Q\|\cdot Lip(f)$. Consequently, we obtain that $\|Q_n\|\leq C^4\|Q\|$. Moreover, it is easy to realise that $Q_n$ is a projection from $\Lip(C^2\cdot B_{\ell_2^n})$ onto $\P({}^2\ell_2^n)$. 

By scaling, we also obtain projections from $\Lip(B_{\ell_2^n})$ onto $\P({}^2\ell_2^n)$ with norms at most $C^4\|Q\|$, for every $n\in\N$. However, for large $n$, this contradicts Theorem \ref{mth: finite-dim} and concludes the proof.
\end{proof}

We already mentioned in the Introduction that there is no hope to extend the conclusion of Theorem \ref{mth: inf-dim} to every infinite-dimensional Banach space. The first result in this direction was that $\P({}^k\ell_1)$ is isomorphic to $\ell_\infty$, \cite{AS}. Its proof goes as follows: one first shows the rather straightforward fact that $\L({}^k \ell_1)$ is isometric to $\ell_\infty$. Then, by the Polarisation formula, $\L^s({}^k \ell_1)$ is complemented in $\L({}^k \ell_1)$ (recall that $\L^s({}^k X)$ is isomorphic to $\P({}^k X)$ for every $X$). The conclusion follows from Lindenstrauss' result that complemented subspaces of $\ell_\infty$ are isomorphic to it, \cite{L}. Alternatively, one could directly use that $\L^s({}^k X)$ is isomorphic to $\L({}^k X)$ when $X^2$ is isomorphic to $X$, \cite{DD}.

\begin{proposition}[\cite{AF}]\label{prop: L1} If $X$ is a separable $\mathscr{L}_1$-space, $\P({}^k X)$ is isomorphic to $\ell_\infty$, for every $k\geq 1$. Hence, $\P({}^k X)$ is complemented in $\Lip(B_X)$.
\end{proposition}
Combining this with the comments below on general polynomials, it also follows that $\P^k(X)$ is isomorphic to $\ell_\infty$. We refer to \cite{LT} for the definition of the $\mathscr{L}_1$-spaces.

\begin{proof} As before, it suffices to embed $\L({}^k X)$ as a complemented subspace of $\ell_\infty$ and appeal to \cite{L}. Since $X$ is a separable $\mathscr{L}_1$-space, there are two uniformly bounded sequences $(\p_j)_{j=1}^ \infty$ and $(\psi_j)_{j=1}^\infty$ of linear maps $\p_j\colon X\to \ell_1^j$, $\psi_j \colon \ell_1^j \to X$ such that $\p_j\circ \psi_j=id$ and $(\psi_j(\ell_1^j)) _{j=1}^\infty$ is an increasing sequence whose union is dense in $X$ (\cite[Proposition II.5.9]{LT}). Define a bounded linear map $\Psi\colon \L({}^k X) \to \left(\sum_{j=1}^\infty \L({^k \ell_1^j})\right)_\infty$ by $\Psi (M):=(M\circ \psi_j)_{j=1}^\infty$; notice that $\left(\sum_{j=1}^\infty \L({^k \ell_1^j})\right)_\infty$ is isometric to $\ell_\infty$.

We also define $\Phi\colon \left(\sum_{j=1}^\infty \L({^k \ell_1^j})\right)_\infty \to \L({}^k X)$ by $(M_j)_{j=1}^ \infty\mapsto \lim_{\mathcal{U}} M_j\circ\p_j$ where $\mathcal{U}$ is a free ultrafilter on $\N$ and the limit is in the pointwise topology. Plainly, $\Phi$ is a bounded linear map. Moreover, it is easy to see that $\Phi \circ \Psi$ is the identity of $\L({}^k X)$ \cite[Lemma 2.3]{AF}. Hence, $\L({}^k X)$ embeds as a complemented subspace of $\ell_\infty$, as desired.
\end{proof}

\begin{remark} Let $X$ be a separable $\mathscr{L}_p$-space, $1<p<\infty$. The same argument as above shows that $\L({}^k X)$ embeds as a complemented subspace of $\left(\sum_{j=1}^\infty \L({^k \ell_p^j})\right)_p$, which is complemented in $\L({^k \ell_p})$. Then, one shows that $\L({^k \ell_p})$ embeds as a complemented subspace of $\L({^k X})$ \cite[Lemma 2.4]{AF} and invokes Pe\l czy\'nski's decomposition method to conclude that $\L({^k X})$ is isomorphic to $\L({^k \ell_p})$, \cite[Theorem 2.1]{AF}.

Let us also point out the similarity between this argument and some techniques from \cite{CCD}. This is also true for the statement of some their results, \emph{e.g.}, \cite[Theorem 3.3]{CCD}.
\end{remark}

In conclusion to our article, we shall mention how to extend our results to $k$-homogeneous polynomials ($k\geq2$) and general polynomials. For the definition of $k$-homogeneous polynomials we refer to Section \ref{sec: poly}. A \emph{polynomial of degree at most $k$} on $X$ is a function $P\colon X\to\R$ of the form $P=\sum_{j=0}^k P_j$, where $P_j\in\P({}^j X)$. The collection of all polynomials of degree at most $k$ is denoted by $\P^k(X)$ and it is a Banach space under the same norm $\|P\|_\P:=\sup_{x\in B_X}|P(x)|$. We denote by $\P^k_0(X)$ the closed subspace comprising polynomials that vanish at $0$, namely those where $P_0=0$ (of course, such restriction is necessary in order to embed into $\Lip(B_X)$).

It is an important fact that $\P^k(X)=\P({}^0X)\oplus \P({}^1X)\oplus\dots\oplus \P({}^kX)$ via the natural decomposition $P=\sum_{j=0}^k P_j\mapsto (P_j)_{j=0}^k$. Moreover, the norm of the projection $P\mapsto P_j$ does not depend on $X$. (This can be shown via Vandermonde matrices as in \cite[Fact 1.1.42]{HJ}, or via the Polarisation formula, \cite[Lemma 1.1.47]{HJ}.) In particular, the map $P\mapsto P\cut_{B_X}$ also defines an isomorphic embedding of $\P^k_0(X)$ into $\Lip(B_X)$.\smallskip

Since $\P({}^2X)$ is complemented in $\P^k_0(X)$, a formal consequence of Theorem \ref{mth: inf-dim} is that $\P^k_0(X)$ is not complemented in $\Lip(B_X)$, whenever $X$ is a Banach space as in Theorem \ref{mth: inf-dim}. For $\P({}^k X)$ the argument is more complicated; the idea is to embed $\P({}^2 X)$ as a complemented subspace of $\P({}^k X)$ \cite[Proposition 5.3]{AS} and mimic the proof of Theorem \ref{mth: finite-dim}. However, for us it will be more convenient to embed $\P({}^2 X)$ as a complemented subspace of $\P({}^k X\oplus\R)$. This depends on the canonical isomorphism $H$ between $\P^k(X)$ and $\P({}^k X\oplus\R)$, given by homogenisation.

More precisely, let $\zeta\colon X\oplus\R\to\R$ be the functional $\zeta(x,t):=t$. The action of $H$ on $P=\sum_{j=0}^kP_j\in\P^k(X)$ is given by $P\mapsto {}^h P:= \sum_{j=0}^k \zeta^{k-j}P_j$. Conversely, every $P\in \P({}^k X\oplus\R)$ has the form $P= \sum_{j=0}^k \zeta^{k-j}P_j$, with $P_j\in \P({}^j X)$, and $H^{-1}(P)= \sum_{j=0}^k P_j$ (roughly speaking, $H^{-1}(P)(x)=P(x,1)$). Moreover, the isomorphism constant between $\P^k(X)$ and $\P({}^k X\oplus\R)$ depends on $k$, but not on $X$.

\begin{theorem}\label{thm: higher order} Let $X$ be a Banach space that contains uniformly complemented $(\ell_2^n)_{n=1}^\infty$. Then, for every $k\geq2$, $\P({}^k X)$ and $\P^k_0(X)$ are not complemented in $\Lip(B_X)$.
\end{theorem}

\begin{proof} We only need to show the assertion concerning $\P({}^k X)$.

Assume that $X$ is a Banach space that contains uniformly complemented $(\ell_2^n)_{n=1}^\infty$ and, towards a contradiction, assume that $\P({}^k X)$ is complemented in $\Lip(B_X)$. The same diagram chasing argument as in the proof of Theorem \ref{mth: inf-dim} shows that there are projections from $\Lip(B_{\ell_2^n})$ onto $\P({}^k\ell_2^n)$ whose norms are bounded uniformly in $n\in\N$. We let $F_n:=\ell_2^n\oplus_\infty \R$. Since $F_n$ is $\sqrt{2}$-isomorphic to $\ell_2^{n+1}$, it also follows that $\P({}^k F_n)$ is complemented in $\Lip(B_{F_n})$, uniformly in $n$.

Next, we exploit the isomorphism $H$. $\P({}^2\ell_2^n)$ is complemented in $\P^k(\ell_2^n)$, uniformly in $n$, and the latter space is isomorphic to $\P({}^k F_n)$, uniformly in $n$ as well. Therefore, the image of $\P({}^2\ell_2^n)$ under $H$ is complemented in $\Lip(B_{F_n})$, uniformly in $n$. However, we shall show that this is not the case.

To simplify the notation, we let $E:=\ell_2^n$ and $F:=F_n=E\oplus_\infty \R$. Moreover, let

$$\P:=H\left(\P({}^2 E)\right)=\left\{P\in\P({}^k F) \colon P=\zeta^{k-2}\cdot P_2, \text{ for some } P_2\in \P({}^2 E) \right\}.$$
Recall that $\zeta$ was defined on $F$ by $\zeta(x,t):=t$ ($x\in E$, $t\in \R$). Heuristically, $\P$ consists of polynomials of the form $t^{k-2}P_2(x)$, for some $2$-homogeneous polynomial $P_2$ on $E$.\smallskip

Our goal will be to show that if $Q$ is any projection from $\Lip(B_F)$ onto $\P$, then the norm of $Q$ diverges with $n$, which leads to the desired contradiction. As in Theorem \ref{mth: finite-dim}, we actually consider the restriction of $Q$ to $\C(B_F)$ and, by averaging, we assume that $Q$ is invariant under $\On$, the group of rotations of the Euclidean space $E$. We then follow the proof of Theorem \ref{mth: finite-dim}, but multiplying the Lipschitz functions there by $\zeta^{k-2}$.

Consider the functions $\p_{ij}:=\psi_{ij}\cdot\zeta^{k-2}$, $\Phi:=\Psi\cdot \zeta^{k-2}$, and $\Phi_d:=\Psi_d\cdot \zeta^{k-2}$. The presence of the factor $\zeta^{k-2}$ now leads to an extra $k/2$ factor in Fact \ref{fact: Lip Psi}, namely we have $Lip(\Phi),Lip(\Phi_d)\leq k/2\e^4$ (indeed, each $p_{ij}$ is $k$-Lipschitz). The invariance argument in {\bf Step 3} proceeds identically to give that $Q(\Phi)=(Q(\Psi))\cdot \zeta^{k-2}$ and analogously for $\Phi_d$. Hence, Claim \ref{claim: max Lip Q Psi and Psi d} remains true with the same estimate. In the estimate of $\alpha$, the unique difference is an extra $(k-1)$ factor; in other words, (\ref{eq: alpha >=}) becomes

\begin{equation}
    \alpha\geq \left(\frac{\pi}{72}-\delta\right)\cdot \big(1-2\e K_n\big)\cdot(k-1).
\end{equation}

Finally, when putting estimates together in {\bf Step 5}, we obtain an inequality stronger than (\ref{eq: last eq with alpha}), since $k/2\leq k-1$. Consequently, we actually even obtain the same lower bound on $\|Q\|$ as the one in Theorem \ref{mth: finite-dim}. We omit further details.
\end{proof}

{\bf Acknowledgements.} We wish to thank the anonymous referee  for carefully reading our manuscript and for the helpful report.

%-----------------------------------------------------------%
%                                                           %
% 						BIBLIOGRAPHY 						%
%                                                           %
%-----------------------------------------------------------%

\end{document}